\documentclass[usenames,dvipsnames, 11pt]{article}
\usepackage[utf8]{inputenc}

\usepackage[margin=2.5cm,a4paper]{geometry} 
\usepackage{multirow,listings,setspace,gnuplottex,latexsym,keyval,ifthen,moreverb,lscape,forest}
\usepackage{pgf,pgfplots,tikz}
\pgfplotsset{compat=1.17}
\usetikzlibrary{arrows}

\usepackage{todonotes}

\usepackage{amsmath, amssymb, amsthm, bm}
\usepackage{thm-restate}
\usepackage{hyperref}
\usepackage[capitalize]{cleveref}
\usepackage{color}
\usepackage{url}
\usepackage{enumitem}
\usepackage{graphicx}
\usepackage{appendix}
\usepackage{cite}
\usepackage{mathtools}
\usepackage{float}

\newcommand{\ee}{\mathbb{E}}
\newcommand{\pp}{\mathbb{P}}

\newcommand{\B}{\mathcal{B}}

\newcommand{\PP}{\mathbb{P}}

\newcommand{\F}{\mathbb{F}}

\newcommand{\PG}{\mathrm{PG}}
\newcommand{\eps}{\varepsilon}

\DeclareMathOperator{\Bin}{Bin}

\newtheorem{thm}{Theorem}[section]
\newtheorem{prop}[thm]{Proposition}
\newtheorem{lem}[thm]{Lemma}
\newtheorem{cor}[thm]{Corollary}

\newtheorem{rem}[thm]{Remark}

\theoremstyle{definition}
\newtheorem{mydef}[thm]{Definition}

\title{The generalized trifference problem}
\author{
Anurag Bishnoi\thanks{Delft Institute of Applied Mathematics, TU Delft. E-mail: {\tt A.Bishnoi@tudelft.nl}}
\and
Bart\l omiej Kielak\thanks{Faculty of Informatics, Masaryk University, Brno, Czech Republic. E-mail: {\tt bkielak@alumni.uj.edu.pl}}
\and
Benedek Kovács\thanks{ELTE Linear Hypergraphs Research Group, Eötvös Loránd University, Budapest, Hungary. Supported by the EKÖP-24 University Excellence Scholarship Program of the Ministry for Culture and Innovation from the source of the National Research, Development and Innovation Fund and the University Excellence Fund of Eötvös Loránd University.
		E-mail: {\tt benoke981@gmail.com}}
	\and Zolt\'an L\'or\'ant Nagy\thanks{ELTE Linear Hypergraphs  Research Group,
		E\"otv\"os Lor\'and University, Budapest, Hungary. The author is supported by the University Excellence Fund of Eötvös Loránd University.	E-mail: {\tt nagyzoli@cs.elte.hu}}
\and G\'abor Somlai\thanks{ELTE Algebra and Number Theory Department,
		E\"otv\"os Lor\'and University and HUN-REN Alfr\'ed R\'enyi Institute of Mathematics Budapest, Hungary. Supported by OTKA grants: FK 142993, STARTING 150576, 138596. E-mail: {\tt gabor.somlai@ttk.elte.hu}}
\and M\'at\'e Vizer\thanks{Department of Computer Science and Information Theory, Budapest University of Technology and Economics, Budapest, Hungary.  Supported by the EXCELLENCE-24 project no.~151504 Combinatorics and Geometry of the NRDI Fund.  E-mail:{\tt vizermate@gmail.com}} 
 \and Zeyu Zheng\thanks{Carnegie Mellon University, Pittsburgh, PA 15213, USA. Supported in part by U.S. taxpayers through NSF grant DMS-2154063. E-mail: {\tt zeyuzhen@andrew.cmu.edu}}}
  \date{}
\begin{document}

\maketitle
\begin{abstract}
    We study the problem of finding the largest number $T(n, m)$ of ternary vectors of length $n$ such that for any three distinct vectors there are at least $m$ coordinates where they pairwise differ. 
    For $m = 1$, this is the classical trifference problem which is wide open. 
    We prove upper and lower bounds on $T(n, m)$ for various ranges of the parameter $m$ and determine the phase transition threshold on $m=m(n)$ where $T(n, m)$ jumps from constant to exponential in $n$. 
    By relating the linear version of this problem to a problem on blocking sets in finite geometry, we give explicit constructions and probabilistic lower bounds. 
    We also compute the exact values of this function and its linear variation for small parameters. 
\end{abstract}

\section{Introduction}

Let $\Sigma$ denote a finite set of symbols, known as the alphabet, where we typically take $\Sigma=\{0,1,2\}$ or a finite field $\mathbb{F}_q$. 
We will refer to subsets of $\Sigma^n$ as codes of length $n$ over $\Sigma$. 
If $a,b,c\in \Sigma^n$
are three codewords, then we say that $a$, $b$ and $c$ \textit{triffer} at position $i$ if $\{a_i,b_i,c_i\}=\{0,1,2\}$. The well-known \textit{trifference problem} refers to the determination of the largest size $T(n)$  of a ternary code $C \subseteq \{0, 1, 2\}^n$ in which for every three distinct codewords $a, b, c \in C$, there is at least one position where they triffer. 
This problem has important applications in computer science \cite{km88, WangXing} and cryptography \cite{barg2010robust, zaverucha2010hash}, and is closely related to many seemingly distant topics, such as strong blocking sets in finite projective spaces \cite{Bishnoi_trifference2024}, minimal codes \cite{Bishnoi_trifference2024, ABNR2022}, and covering codes \cite{heger2021short}. Recently, the problem has gained significant interest, in part due to these connections, which has led to
various improvements in the known bounds \cite{Bishnoi_trifference2024, Kurz2024, bhandari2025improved, pohoata2022trifference, della2025bounds}.

The trifference problem is a special case of the following perfect $k$-hashing problem.
Let $k\ge 3$ be a fixed integer. A code $C$ over an alphabet $\Sigma$ is called a
perfect $k$-hash code if for every $k$-tuple of distinct codewords in $C$, there exists a position where all these elements differ.
 Finding optimal upper bounds for the maximum size of perfect $k$-hash codes is a fundamental open problem in theoretical computer
science. 
The cornerstone result is based on
 an  elementary double counting argument due to K\"orner \cite{Korner1973}, and it states that  $|C|\le (k-1)(1+\frac{1}{k-1})^n.$  
 This upper bound was improved by Fredman and Komlós \cite{fredman1984size}, and later, by Arikan \cite{arikan1994upper}.
 In recent years, there have been further improvements by Dalai, Guruswami and Radhakrishnan \cite{dalai2019improved}, and most recently by Guruswami and Riazanov \cite{guruswami2022beating} in the cases when $k>3$. 
 However, improving the trivial upper bound in the case $k=3$ remained a notoriously difficult problem. Very recently, Bhandari and Khetan \cite{bhandari2025improved} obtained the following polynomial improvement on the upper bound, 
 \[T(n)\leq cn^{-2/5}\left(\frac32\right)^n,\] where $c$ is an absolute constant. 
 On the other hand, the best lower bound on $T(n)$ is $(\frac{9}{5})^{n/4}$ due to Körner and Marton \cite{km88} and dates back to 1988 (see \cite{Bishnoi_trifference2024} for an alternate proof of the same bound).

In this paper, we study the following natural generalization of the trifference problem. 
\begin{mydef}
Suppose that $n\ge m\ge 1$ are integers. Let $T(n, m)$ be the largest size of a ternary code $C \subseteq \{0, 1, 2\}^n$ such that for any three distinct codewords $x, y, z$ in $C$, we have at least $m$ coordinates where they triffer. A code with such a property will be referred to as an \textit{$m$-trifferent code}.
\end{mydef}

Recall that the \textit{Hamming distance} $d(u,v)$ between two codewords $u$, $v$ is the number of coordinates where $u$ and $v$ differ. The expression $A_{q}(n,d)$ denotes the maximum number of possible codewords in a $q$-ary block code of length $n$ and minimum Hamming distance $d$, see \cite{macwilliams1977theory}. The determination of  (the asymptotic behavior of)  $A_{q}(n,d)$ is one of the most central problems in coding theory.
 
The determination of the maximum size of generalized trifferent codes 
 $T(n, m)$ can be interpreted as the analogue of maximum size codes with given Hamming distance, where $m$ corresponds to the `minimum distance' in this case, thus highlighting its relevance.
 This generalization has been studied  by Bassalygo,  Burmester, Dyachkov, Kabatianski in \cite{Bassalygo1997HashCodes} under the name of $k$-hash codes with distance $d_k$ over an alphabet of size $q$. 
 Our problem corresponds to the special case of $q = k = 3$ and $d_3 = m$. 

Our main contribution is to give both lower and upper bounds on $T(n,m)$ for various ranges of $m$. 
We first show that there is a phase transition around the value $m\approx \frac29 n$, which is analogous to the phase transition of error-correcting codes.
This was already stated implicitly in \cite{Bassalygo1997HashCodes} for the general problem, but we also include a proof, and we significantly improve their bounds for $q=k=3$.

\begin{thm}\label{main} Suppose that $\eps>0$ is a real number. If $m>(\frac29+\eps)n$, then $T(n,m)$ is bounded above by a constant depending only on $\eps$. On the other hand, if $m<(\frac29-\eps)n$, then $T(n,m)$ is bounded from below by an exponential function $(1+\delta)^n$, where $\delta>0$ depends only on $\eps$.
\end{thm}

Before stating our other main results, let us first state some elementary observations on the function in view. 
\begin{prop}\label{trivi} \begin{enumerate}[label=(\roman*)]
    \item\label{labeli} $T(n, n) = 3$ for every $n\ge 1$.
    \item\label{labelii} $T(n, m) \leq T(n - 1, m - 1)$ for every $n\ge m\ge 2$. 
    \item $T(n, m) \geq T(n/m, 1)$ for every $n\ge m\ge 1$ with $m\mid n$. 
    \item $\frac23T(n,m)\leq T(n-1, m)$ for every $n>m\ge 1$.
\end{enumerate}
\end{prop}
\begin{proof} $(i)$ is straightforward. To see $(ii)$, observe that  if we erase a coordinate of an $m$-trifferent code, we obtain an $(m-1)$-trifferent code. To prove $(iii)$, we take the $m$-fold repetition of each codeword in a $1$-trifferent code of length $n/m$, giving an $m$-trifferent code of length $n$. Finally for $(iv)$, if we fix a coordinate $i$ in an $m$-trifferent code and take all the codewords whose $i$-th coordinate is not equal to $j$ for an arbitrarily chosen $j$, we again obtain an $m$-trifferent code. If we then erase the $i$-th coordinate, the resulting code is still $m$-trifferent. Taking $j$ to be the least frequent value appearing at coordinate $i$, we get the desired bound.
\end{proof}

Using the upper bound on $T(n)=T(n,1)$ from \cite{bhandari2025improved}, Proposition \ref{trivi} \ref{labelii} applied $m-1$ times immediately gives the following. 
\begin{cor}
\label{cor:triv_upper}
   $T(n,m)\leq c(n-m+1)^{-2/5}(\frac32)^{n-m}$ holds, where $c$ is an absolute constant.
\end{cor}

In \cite{Bassalygo1997HashCodes}, it was shown that the elementary recursive bound can be improved to give 
\begin{equation}
\label{eq:old_bound}
T(n, m) \leq m \left( \frac{3}{2}\right)^{n - \frac{9}{2}(m -1)}.
\end{equation}
We provide significant improvements on these upper bounds in various ranges of values for $m$. 

\begin{thm}[Upper bound] \label{mainupper}Let $m$ be of the form $m=3\ell+r$ for $r\in \{0,1,2\}$. Then  \[
    T(n,m)\le \begin{cases}
\left(\frac{3}{2}\right)^n\frac{6(3\ell+2)2^{2\ell+1}}{\binom{n}{2\ell+1}} &\text{if $r =2 $} \\[10pt]
\left(\frac{3}{2}\right)^n\frac{3(3\ell+1)2^{2\ell}}{\binom{n}{2\ell}} &\text{if $r=1$} \\[10pt]
\left(\frac{3}{2}\right)^n\frac{6\ell2^{2\ell-1}}{\binom{n}{2\ell-1}} &\text{if $r=0$.}\end{cases} 
    \]
\end{thm}

Note that for constant $m$ this bound gives a polynomial improvement over Corollary~\ref{cor:triv_upper} and Equation~(\ref{eq:old_bound}), whereas for $m = \Theta(n)$ it gives an exponential improvement. 
For small values of $m$, we obtain the following bounds:
\begin{thm}\label{m_small}
    $$ T(n, 3) \leq T(n,2)\le \left(\frac{3}{2}\right)^n\frac{10}{n}, \ \ \  T(n,4)\le \left(\frac{3}{2}\right)^n\frac{44}{n^2}, \ \  T(n,5)\le \left(\frac{3}{2}\right)^n\frac{232}{n^3}.$$
\end{thm}

If $m$ is of order $\Theta(n)$, the bounds of Theorem \ref{mainupper} can be further improved and provide asymptotic improvements to (\ref{eq:old_bound}):

\begin{thm}[Refined upper bound when $m$ is large]\label{upper_mu} Let $m=\left\lceil\lambda n\right\rceil$ for some $\lambda \in (0, \frac29)$, and $\hat{x}$ be the smallest real root of the cubic polynomial
 ${\lambda-\frac{3}{2}x(1-x)^2}$.
 
Then 
     \[
    T(n,m)\le \left(\frac{3}{2}\right)^n\frac{C}{2^{ -\hat{x} n }\binom{n}{ \lceil \hat{x}n\rceil-1 }}, 
    \] with a suitable constant $C=C(\lambda)$.
\end{thm}

As far as lower bounds are concerned, we prove the following in the same range of parameters.

\begin{thm}[Lower bound]\label{mainlower} For a fixed value $0<\lambda<\frac29$, let $m=\left\lfloor \lambda n\right\rfloor$. Then there exists a constant $c>0$ such that
$$T(n,m)\ge c\left(\frac92\lambda\right)^{\frac12\lambda n}\left(\frac97(1-\lambda)\right)^{\frac12(1-\lambda)n}$$
    for $n$ large enough.
\end{thm}

The paper is organized as follows. In Section \ref{sec:notation} we introduce the necessary notation. We present the proofs of Theorems \ref{main}, \ref{mainlower}, \ref{mainupper} and \ref{upper_mu} in Section \ref{sec:mainsec}. Then we study the linear version of the problem in Section \ref{sec:linearsec}.  We describe the finite geometric analogues of $m$-trifferent linear codes and their relation to minimal codes. Using probabilistic constructions, we obtain lower bounds on linear $m$-trifferent codes $T_L(n,m)$. 
Finally, in Section \ref{sec:concluding}, we present computational results on $T(n,m)$ and $T_L(n,m)$, and discuss the relation between the sunflower conjecture and the trifference problem.

\section{Notation and preliminaries}\label{sec:notation}

For a codeword $c\in \Sigma^n$, $\sigma(c)$ denotes the \textit{support} of the codeword, i.e., the set of indices where $c$ is non-zero. The notation $[n]$ denotes the integer interval $\{x\in \mathbb{Z}: 1\le x\le n\}$.

\begin{mydef}
    A linear code $C \subseteq \mathbb{F}_q^n$ is called \textit{minimal} if for any two linearly  independent $c, c' \in C$, we have $\sigma(c) \not \subseteq \sigma(c')$. 
    A linear code $C \subseteq \mathbb{F}_q^n$ is called \textit{$m$-minimal} if after the deletion of any $m - 1$ coordinates, all codewords remain distinct and the remaining code is still minimal.
\end{mydef}

Logarithms in this paper are taken base $2$.

\begin{mydef}
The \textit{binary entropy function} $H(p)$ is defined as
$$H(p)=-p\log p-(1-p)\log(1-p)$$
\noindent for $0\le p\le 1$.     
\end{mydef}

We will use the following well-known bound on binomial coefficients:

\begin{lem}
$$\frac{1}{n+1}2^{nH\left(\frac{k}{n}\right)}\le \binom{n}{k}\le 2^{nH\left(\frac{k}{n}\right)}.$$
\end{lem}
\noindent For a refined bound, see \texorpdfstring{\cite[Lemma 17.5.1, p.~666.]{infotheory}}.

The\textit{ relative entropy} or Kullback-Liebler divergence is defined as $$D_\text{KL}(P \parallel Q) = \sum_{ x \in \mathcal{X} } P(x) \log\left(\frac{P(x)}{Q(x)}\right)$$ for a pair of discrete probability distributions $P$ and $Q$ defined on the same sample space $\mathcal{X}$. In the special case when $P$ and $Q$ correspond to the distribution of indicator variables with appearance probability $p$ and $p'$, we use the simpler notation $H(p, p')$, which simplifies to

$$H(p, p')= p\log \frac{p}{p'}+(1-p)\log\frac{1-p}{1-p'}. $$

\begin{mydef}
    For any $S \subseteq \{0,1,2\}^n$, we define $T_S(n, m)$ to be  the size of the largest $m$-trifferent set contained in $S$.
\end{mydef}
 The lemma below follows the ideas of the proof in \cite{bhandari2025improved} for the similar bound on $T_S(n)$. This statement will enable us to give upper bounds on $T(n,m)$ via bounding the maximum  size of $m$-trifferent sets contained in certain subsets of the space $\Sigma^n$.

\begin{lem}\label{S-set}
For any subset $S \subseteq \{0,1,2\}^n$,
\[ \frac{T(n, m)}{3^n} \leq \frac{T_S(n, m)}{|S|}.\]
\end{lem}
\begin{proof}
    Let $C$ be an $m$-trifferent code in $\mathbb{F}_3^n$, identified with $\{0,1,2\}^n$. 
    For a vector $v \in \mathbb{F}_3^n$, let $S_v = S + v = \{x + v : x \in S\}$. 
    We double count the size of the set $\{(x, v) : x \in C \cap S_v\}$. 
    Since $x \in C$ is contained in $S_v$ if and only if $v = x - s$ for some $s \in S$, each element of $C$ is contained in $|S|$ many $S_v$'s.
    Therefore,
    \[\sum_{v \in \mathbb{F}_3^n} |C \cap S_v| = |C| |S|.\]
    By averaging, there must exist a vector $u$ such that 
    $|C \cap S_u| \geq |C| |S|/3^n$. 
    We know that $C \cap S_u$ is an $m$-trifferent set inside $S_u$, and since $S_u$ is just a translate of $S$, $|C \cap S_u| \leq T_S(n, m)$, giving us the inequality. 
\end{proof}
Notice that the proof of the previous lemma uses the fact that the set of subsets of $\F_3^n$ that triffer is closed under translation. 
\section{Proof of the lower and upper bounds}\label{sec:mainsec}

We start by proving the lower bound on $T(n,m)$.

\begin{thm}\label{thm:probconstr}
Suppose that $1\le m<\frac29n$. Then $T(n,m)\ge \frac{2\sqrt{2}}{3}\cdot 2^{nh/2}-1$, where $h=H\left(\frac{m}{n}, \frac29\right)$.

\end{thm}
\begin{proof} We use the alteration method.
    For a positive integer $t$ fixed later on, choose codewords $c_1, c_2, ..., c_t\in \{0,1,2\}^n$ independently and uniformly at random, and let $C_0=\{c_1,c_2,...,c_t\}$. A triple of codewords $\{c_x,c_y,c_z\}$ is called \textit{bad} if they triffer at less than $m$ positions. Let $B$ denote the number of bad triples in $C_0$. Clearly $\ee[B]=\binom{t}{3}p_b$ where $p_b=\pp[\{c_1,c_2,c_3 \} \mathrm{ \ is \ bad}]$. Create a new code $C$ from $C_0$ by deleting one of the codewords from each bad triple. Then, $C$ will be $m$-trifferent, and $|C|\ge t-B$. Therefore, $T(n, m)\ge \ee[|C|]\ge t-\binom{t}{3}p_b$.

    If $c_1$, $c_2$ and $c_3$ are uniformly random elements of $\{0,1,2\}^n$, then for each $1\le i\le n$, the probability that the codewords triffer at position $i$ is $\frac{3!}{3^3}=\frac29$.
    Therefore, if $X$ denotes the number of positions where the three codewords triffer, then $X$ follows the binomial distribution $\Bin\left(n,\frac29\right)$. By the well-known Chernoff bound, $p_b\le \pp[X\le m]\le 2^{-nh}$, where $h=H\left(\frac{m}{n}, \frac29\right)$ satisfies $0<h<\log\frac97$.

    Thus, if we choose $t=\lfloor \sqrt{2}\cdot 2^{nh/2}\rfloor$, then we get \begin{equation*}
    T(n,m)\ge t-\binom{t}{3}p_b\ge t-\frac16t^3p_b\ge \sqrt{2}\cdot 2^{nh/2}-1-\frac{\sqrt{2}}{3}\cdot 2^{3nh/2}2^{-nh}=\frac{2\sqrt{2}}{3}\cdot 2^{nh/2}-1. \qedhere
    \end{equation*}
\end{proof}

\begin{rem}\label{rem:probconstr_r}
 (i) In the case where $m$ is a function of $n$ with $m=o(n)$ (for example $m\ge 1$ is a constant) and $n\to\infty$, we get that for some $c>0$,~ $T(n,m)\ge c\left(\frac{3}{\sqrt{7}}\right)^n$ for all $n$. Note that $\frac{3}{\sqrt{7}}\approx 1.1339$.

 (ii) If $m=\left\lfloor\lambda n\right\rfloor$ for a fixed value $0<\lambda<\frac29$, and $n\to \infty$, we get that for some $c>0$, $T(n,m)\ge c\cdot 2^{\frac12H(\lambda,\frac29)n}=cb(\lambda)^n$, where
$$b(\lambda)=2^{\frac12H\left(\lambda,\frac29\right)}=\left(\frac92\lambda\right)^{\frac12\lambda}\left(\frac97(1-\lambda)\right)^{\frac12(1-\lambda)}$$

is a strictly decreasing function for $\lambda\in \left[0,\frac29\right]$ with $b(0)=\frac{3}{\sqrt{7}}$ and $b\left(\frac29\right)=1$.
\end{rem}
\noindent This completes the proof of Theorem \ref{mainlower}, which also implies Theorem \ref{main} when $m<(\frac29-\eps)n$. Next, we complete the proof of Theorem \ref{main} by considering the case $m>(\frac29+\eps)n$. Note that the following result is also a consequence of Theorem 4 in \cite{Bassalygo1997HashCodes}.

\begin{thm}\label{largem}
    For any fixed $\frac29<\lambda\le 1$, there exists a constant $c(\lambda)$ such that $T(n, \left\lceil \lambda n\right\rceil) \le c(\lambda)$ for all $n$.
\end{thm}
\begin{proof}
    Let $C=\{c_1,c_2,...,c_T\}\subseteq \{0,1,2\}^n$ be a code for which any three codewords triffer in at least $\lambda n$ positions.

    Let $K$ denote the number of quadruples $(x,y,z,i)$ such that:
    \begin{itemize}
        \itemsep0em
        \item $1\le x<y<z\le T$ and $1\le i\le n$ are integers,
        \item and codewords $c_x$, $c_y$ and $c_z$ triffer at position $i$.
    \end{itemize}
    
    We will bound $K$ from below and above. Firstly, we can see that for every triple $(x,y,z)$ with $x<y<z$, we have at least $\lambda n$ positions where the three codewords triffer, so $K\ge \binom{T}{3}\cdot \lambda n$.

    On the other hand, for a fixed value $1\le i\le n$, let us put an upper bound on the number of triples triffering at position $i$. If the number of codewords in $C$ having $i$-th coordinate $0$, $1$ and $2$ is $r_i$, $s_i$ and $t_i$ respectively, then the following hold:
    \begin{itemize}
        \itemsep0em
        \item $r_i,s_i,t_i\ge 0$,
        \item $r_i+s_i+t_i=T$,
        \item the number of triples triffering at position $i$ is $r_is_it_i$.
    \end{itemize}
    By the AM-GM inequality, $r_is_it_i\le \left(\frac{r_i+s_i+t_i}{3}\right)^3=\frac{1}{27}T^3$. Therefore,

    $$K=\sum_{i=1}^n r_is_it_i\le \frac{1}{27}nT^3.$$

    This means that $\binom{T}{3}\cdot \lambda n\le \frac{1}{27}nT^3$, giving

    $$\frac{T(T-1)(T-2)}{T^3}\lambda \le \frac{2}{9}.$$

    Since $\frac{T(T-1)(T-2)}{T^3}\to 1$ as $T\to \infty$ and $\frac{2}{9}/\lambda<1$, we conclude that $T$ is bounded by a~constant depending only on $\lambda$.
\end{proof}

We continue by proving upper bounds on $T(n,m)$. First, we present a consequence of the threshold value $\lambda=\frac29$ and of Proposition \ref{trivi}. This is similar to Theorem 5 of \cite{Bassalygo1997HashCodes}.

\begin{thm}\label{derived} Let $0<\alpha<\frac29$ and $0<\eps<\frac79$ be real numbers. Then
$$T\left(n,\left\lceil \left(\frac29-\alpha\right)n\right\rceil\right)\le c\cdot \left(\frac32\right)^{4.5(\alpha+\eps)n}$$ for some constant $c=c(\eps)$.
\end{thm}

\begin{proof} For $m=\left\lceil \left(\frac29-\alpha\right) n\right\rceil$, Proposition \ref{trivi} (iv) can be applied $t$ times iteratively, where $$t=\left\lceil n-\frac{1}{\frac29+\eps}m\right\rceil.$$

We hence get $T(n,m)\le \left(\frac32\right)^t T(n-t,m)$ and it follows from the choice of $t$ that $m\ge \left(\frac29+\eps\right)(n-t)$, so by Theorem \ref{largem} we have that $T(n-t,m)$ is at most a constant $c'=c'(\eps)$. Altogether, $$T\left(n, \left\lceil \left(\frac29-\alpha\right) n\right\rceil\right)\le c'\cdot \left(\frac32\right)^{n-\frac{1}{\frac29+\eps}m+1}\le c\cdot \left(\frac32\right)^{n\left(1-\frac{2/9-\alpha}{2/9+\eps}\right)} \le c\cdot \left(\frac32\right)^{\frac92(\alpha+\eps)n},$$

where $c(\eps)=\frac32c'(\eps)$.
\end{proof}

We now look at the other regime where $m$ is small. 
We first prove a better upper bound than $T(n, 2) \leq T(n, 1) \leq c n^{-2/5} (3/2)^n$.

\begin{thm}\label{m=2}
    \[T(n, 2) \leq \frac{10}{n} \left(\frac32\right)^n.\]
\end{thm}
\begin{proof}
    We use the setup of Lemma \ref{S-set} for the set $S = \{x \in \{0, 1, 2\}^n : \# \{i : x_i = 2\} = 1\}$. Let $C\subseteq S$ be a $2$-trifferent code. Observe that if $x_i = y_i = 2$ for some $x, y \in S$ and $i \in [n]$, then for any $z \in S$ the triple $(x,y,z)$ would triffer in at most one position. Hence, for any $i \in [n]$ there exists at most one $x \in C$ such that $x_i = 2$.

    Fix some $t\in C$ and assume that $t_1 = 2$. The remaining codewords in $C$ start with $0$ or $1$ so if $|C| \geq 6$, then by pigeonhole principle there exist three codewords $x, y, z \in C$ such that $x_1 = y_1 = z_1$. 
    Without loss of generality, $t = (2, *, \cdots, *)$, $x = (0, 2, *, \cdots, *)$, $y = (0, *, 2, *, \cdots, *)$, and $z = (0, *, *, 2, *, \cdots, *)$, where each $*$ is either a $0$ or $1$. 
    By considering the triples $t, x, y$ and $t, x, z$, we get $y_2 = z_2$, and from the triples $t, x, y$ and $t, y, z$ we get $x_3 = z_3$. 
    Therefore, $\{x_i, y_i, z_i\} = \{0, 1, 2\}$ is only possible at $i = 4$, so $x$, $y$ and $z$ triffer in at most $1$ position, contradicting the $2$-trifference property.

    Hence $|C|\le 5$, and by Lemma \ref{S-set} we get \begin{equation*}
    T(n,2)\le \frac{5\cdot 3^n}{n\cdot 2^{n-1}}=\frac{10}{n}\left(\frac32\right)^n.\qedhere
    \end{equation*}
\end{proof}

Before proving the main upper bound (Theorem \ref{upper_mu}), we present an upper bound in the spirit of the bound on $T(n,1)$ given by Bhandari and Khetan \cite{bhandari2025improved}, using extremal graph theoretic tools for an auxiliary graph assigned to the trifferent code.

\begin{thm} \label{aux_graph}  
    \[ T(n,2\ell+1) \leq  \left( \frac 32 \right)^n \frac{\binom{n}{\ell}}{\binom{n}{2\ell}} \cdot 4(2\ell+1). \] 
\end{thm}
\begin{proof}
    We now use the setup of Lemma \ref{S-set} for $m=2\ell+1$ and $S = \{x \in \{0, 1, 2\}^n : \# \{i : x_i = 2\} = 2\ell\}$.
    We wish to bound $T_S(n, 2\ell + 1)$ from above. To do this, we will define a bipartite graph $G$ on the two vertex classes $U$ and $V$, where $U$ and $V$ are disjoint sets with $|U|=|V|=\binom{n}{\ell}$, both of whose elements are labelled with the $\ell$-element subsets of $[n]$. Fix any $(2\ell+1)$-trifferent code $C \subseteq S$. For each codeword $x \in C$, there exists a~set of indices $I \subseteq [n]$ of size $2m$ such that $x_i = 2$ for each $i \in I$ and $x_i \neq 2$ otherwise. For each partition $I = I_1 \cup I_2$ with $|I_1| = |I_2| = \ell$, we add an edge $(I_1, I_2)$ to $G$, where $I_1 \in U$ and $I_2 \in V$. This way, we add $\binom{2\ell}{\ell}$ edges for each $x \in C$. Moreover, since $C$ is $(2\ell+1)$-trifferent, each edge corresponds to at most two different codewords from $C$. Therefore, $T_S(n, 2\ell+1)$ is at most $\frac{2}{\binom{2\ell}{\ell}} |E(G)|$.\\
    We claim that $G$ does not contain a $P_4$ (a four-vertex path). Assume it is not the case and $(I_1, I_2)$, $(I_3, I_2)$, and $(I_3, I_4)$ are edges of some $P_4$ (where $I_1, I_3 \in U$ and $I_2, I_4 \in V$) and these edges correspond to some codewords $x$, $y$, $z \in C$. Then, since $I_1 \cup I_2 \neq I_2 \cup I_3$, we must have $x \neq y$; analogously, $y \neq z$. Moreover, since $I_2 \neq I_4$ and both are of the same size and their union is disjoint from $I_3$, there exists $i \in I_2 \setminus (I_3 \cup I_4)$. In particular, $x_i = 2 \neq z_i$, which implies $x \neq z$, hence $x$, $y$, $z$ are pairwise different. Since $x,y,z$ are codewords from a $(2\ell+1)$-trifferent code, there must exist a~set $J \subseteq [n]$ of size at least $2\ell+1$ such that $\{x_i, y_i, z_i\} = \{0,1,2\}$ for each $i \in J$. However, it is easy to see that $J$ must be a~subset of $I_1 \cup I_4$, but on the other hand, $I_1 \cup I_4$ consists of at most $2\ell$ elements, which is a~contradiction.\\
    Since $G$ is $P_4$-free (and also triangle-free by definition), it must be a~star forest, hence it contains at most $|V(G)| - 1 < 2\binom{n}{\ell}$ edges. This implies $T_S(n, 2\ell+1) \leq  \frac{4}{\binom{2\ell}{\ell}} \cdot \binom{n}{\ell}$ and \begin{equation*} T(n,2\ell+1) \leq 3^n \cdot \frac{T_S(n,2\ell+1)}{|S|} \leq 3^n \cdot \frac{\frac{4}{\binom{2\ell}{\ell}} \cdot \binom{n}{\ell}}{\binom{n}{2\ell}2^{n-2\ell}} \leq \left( \frac 32 \right)^n \frac{\binom{n}{\ell}}{\binom{n}{2\ell}} \frac{4\cdot 2^{2\ell}}{\binom{2\ell}{\ell}} \leq \left( \frac 32 \right)^n \frac{\binom{n}{\ell}}{\binom{n}{2\ell}} \cdot 4(2\ell+1). \qedhere \end{equation*}
\end{proof}

Next, we prove Theorem \ref{upper} and \ref{upper_mu} and compare them with the upper and lower bounds above.

\begin{thm}\label{upper} Suppose that $m\le \frac29n.$ Then \[
    T(n,m)\le \left(\frac{3}{2}\right)^n\frac{3m\cdot 2^{\lceil 2m/3\rceil }}{\binom{n}{\lceil 2m/3\rceil -1}}.
    \]

\end{thm}

\begin{proof}[Proof of Theorem \ref{upper} and \ref{upper_mu}] In the setup of Lemma \ref{S-set}, suppose that $S$ consists of the codewords with exactly $k<\frac13n$ many $2$'s, where $k$ is an integer to be fixed later. Let $C=\{c_1, c_2, ..., c_T\}\subseteq S$ be an $m$-trifferent code.
    To bound $|C|=T$, we double count the number of quadruples $$K=|\{(x,y,z,i) : 1\le x<y<z\le T, ~i\in [n], \text{~$c_x$, $c_y$ and $c_z$ triffer at position $i$}\}|$$ along the lines of the proof of Theorem \ref{largem}. Then $K\ge m\binom{T}{3}$. Let $r_i, s_i, t_i$ denote the number of codewords in $C$ having $i$-th coordinate $0$, $1$, and $2$, respectively. Then $\sum_{i=1}^n r_is_it_i=K$ and $\sum_{i=1}^n t_i=kT$. Note that 
  \begin{equation}\label{eq:doublecount}
        m\binom{T}{3}\le\sum_{i=1}^n r_is_it_i \le \sum_{i=1}^n t_i\left(\frac{T-t_i}{2}\right)^2 \le n\frac{kT}{n} \left(\frac{T-\frac{kT}{n}}{2}\right)^2
        \le  \frac{kT^3}{4}.      
  \end{equation}

Here, the use of Jensen's inequality in the step $\sum_{i=1}^n t_i\left(\frac{T-t_i}{2}\right)^2 \le n\frac{kT}{n} \left(\frac{T-\frac{kT}{n}}{2}\right)^2$ is justified only in the case when $t_i\leq \frac23T$ for every $i$, as the function $x(T-x)^2$ is concave in $[0, \frac23T]$. However, the maximum value of the left hand side over $(t_i)\in \mathbb{Z}\cap [0, T]$ cannot be taken  when there exists an index $i$ for which $t_i>\frac23T$. Indeed, the sum $\sum_{i=1}^n t_i=kT$ implies that there exists an index $j$ with $t_j<\frac13T$ in view of the assumption on $k$. But then altering $t_i\rightarrow t_i-1$, $t_j\rightarrow t_j+1$ for this index pair, we would obtain an increment in the sum, as  $$(t_i-1)\left(\frac{T-t_i+1}{2}\right)^2+(t_j+1)\left(\frac{T-t_j-1}{2}\right)^2 - t_i\left(\frac{T-t_i}{2}\right)^2 - t_j\left(\frac{T-t_j}{2}\right)^2=$$ $$\frac14 (4 T - 3 t_i - 3 t_j) (-1 + t_i - t_j)>0,$$ since both factors $(4 T - 3 t_i - 3 t_j)=((3 T - 3 t_i)+(T - 3 t_j))$ and $ (-1 + t_i - t_j)$ are positive.

Now let us    simplify the inequality \eqref{eq:doublecount}:   $m\binom{T}{3}\le  \frac{kT^3}{4}$. By expanding the formula, we obtain \begin{equation}\label{eq:km}
    (2m-3k)T^2-6mT+4m\le 0.        
    \end{equation}
 Plugging in $k=\lceil 2m/3\rceil -1$ gives us the bound \[
    T\le \begin{cases}
6m &\text{if $m \equiv 2\pmod 3 $}\\
3m &\text{if $m \equiv 1\pmod 3$}\\
2m &\text{if $m \equiv 0\pmod 3$.}
\end{cases} 
    \] Hence, if $m$ is of form $m=3\ell+r$ for $r\in \{0,1,2\}$, then Lemma \ref{S-set} implies \begin{equation}\label{eq:up}
    T(n,m)\le 3^n\cdot \frac{T_S(n,m)}{\binom{n}{k}\cdot 2^{n-k}}\le \begin{cases}
\left(\frac{3}{2}\right)^n\frac{6(3\ell+2)2^{2\ell+1}}{\binom{n}{2\ell+1}} &\text{if $r =2 $}\\[10pt]
\left(\frac{3}{2}\right)^n\frac{3(3\ell+1)2^{2\ell}}{\binom{n}{2\ell}} &\text{if $r=1$}\\[10pt]
\left(\frac{3}{2}\right)^n\frac{6\ell2^{2\ell-1}}{\binom{n}{2\ell-1}} &\text{if $r=0$,}\end{cases}    
    \end{equation}
as claimed in Theorem \ref{mainupper} and its more compact form Theorem \ref{upper}. Note that the same argument can be slightly refined, as Inequality \eqref{eq:km} implies a slightly better upper bound on $T$. Indeed, for small values of $m$, we get the exact  bounds presented below. 
\begin{center}
\begin{tabular}{c|ccccc}
$m$ & $2$ & $3$ & $4$ & $5$ & $6$\\\hline
$k$ & $1$ & $1$ & $2$ & $3$ & $3$\\\hline
upper bound on $T$ & $2(3+\sqrt{7})$ & $3+\sqrt{5}$ & $2(3+\sqrt{7})$ & $15+\sqrt{205}$ & $2(3+\sqrt{7})$\\\hline
approx. upper bound on $T$ & $\approx11.29$ & $\approx 5.24$ & $\approx11.29$ & $\approx 29.32$ & $\approx11.29$
\end{tabular}
\end{center}
    These bounds on $T$ together with Inequality \eqref{eq:up}, and the fact that $T$ is an integer, give the upper bounds of Theorem \ref{m_small}.\\

    Now we refine the argument to derive Theorem \ref{upper_mu}. Suppose now that $m=\left\lceil \lambda n\right\rceil$ for some $0<\lambda<\frac29$. Let $\hat{x}$ be the smallest real root of $\lambda =\frac32x(1-x)^2$. From Inequality \eqref{eq:doublecount}, we then get $$m\binom{T}{3}\le {kT} \left(\frac{T-\frac{kT}{n}}{2}\right)^2=\frac{T^3}{4}k\left(1-\frac{k}{n}\right)^2.$$
    Take $k= \mu n:= \lceil \hat{x}n\rceil-1.$ Then we get 

   $$\lambda(T-1)(T-2)\leq \frac{3T^2}{2}\mu\left(1-\mu\right)^2,$$
   which is equivalent to
   $$\left(\lambda-\frac{3}{2}\mu(1-\mu)^2\right)T^2\leq \lambda(3T-2).$$
   Consequently, $T< \frac{3\lambda}{\lambda-\frac{3}{2}\mu(1-\mu)^2}$, provided that $\lambda-\frac{3}{2}\mu(1-\mu)^2$ is positive, which holds in view of the choice of $\mu$.
   Then, similarly to the previous case, we obtain that 
     \[
    T(n,m)\le 3^n\cdot \frac{T_S(n,m)}{\binom{n}{k}\cdot 2^{n-k}}\le\left(\frac{3}{2}\right)^n\frac{C}{2^{ -\hat{x} n }\binom{n}{ \mu n }}, 
    \] with a suitable constant $C=C(\lambda)$.\\ Here, note that $\hat{x}/\lambda =\hat{x}(\lambda)\lambda^{-1}$ is monotone increasing in $\lambda\in (0, 2/9)$ for fixed $n$, with $\hat{x}/\lambda\geq  2/3$ on the whole interval while $\hat{x} n/m\to\frac 32$ as $\lambda \to 2/9$ and $n\to \infty$. On the other hand, $\mu$ is a good approximation of $\hat{x} $ as $|\mu-\hat{x}|\le n^{-1}$.
\end{proof}

\begin{figure}[H]
\centering
\includegraphics*[]{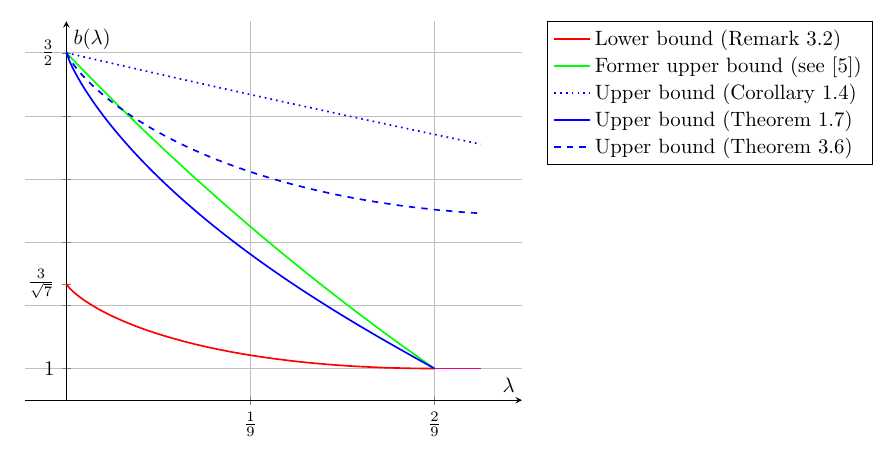}
\caption{Graph indicating various bounds of the form $(b(\lambda)(1\pm o(1))^n$ on $T(n,\lambda n)$}
\label{fig:expbounds_withgraph}
\end{figure}

\begin{rem}
In Figure \ref{fig:expbounds_withgraph}, we compare the lower bound obtained from the probabilistic construction in Remark \ref{rem:probconstr_r}, with the former upper bound and three new upper bounds from Corollary \ref{cor:triv_upper} and Theorems \ref{upper_mu} and \ref{aux_graph}, for a fixed $0\le \lambda\le\frac29$. For each $\lambda$, the values $b(\lambda)$ are plotted where $T(n,\left\lfloor\lambda n\right\rfloor)\ge (b(\lambda)(1-o(1))^n$ for lower bounds and $T(n,\left\lceil\lambda n\right\rceil)\le (b(\lambda)(1+o(1))^n$ for upper bounds. Because of Theorem \ref{largem} we know that $b(\lambda)=1$ is the correct value for $\frac29<\lambda\le 1$.
\end{rem}

\section{The linear version}\label{sec:linearsec}

The equivalence between linear trifferent codes and minimal codes was first showed in \cite{Bishnoi_trifference2024}. Here, we establish an analogous equivalence between linear $m$-trifferent codes and $m$-minimal codes, which allows us to leverage the rich theory of minimal codes to study linear trifferent codes.

\begin{lem}
A linear code $C$ in $\mathbb{F}_3^n$ is $m$-trifferent if and only if it is $m$-minimal.  
\end{lem}
\begin{proof}
    Suppose $C$ is $m$-trifferent but not $m$-minimal. Then there exist non-zero codewords $x$ and $y$ such that after deleting some set of $m-1$ coordinates, $\sigma(x) \subsetneq  \sigma(y)$. 

    Consider the codeword $z = - x$, which must be in $C$ by linearity. Let $i$ be a coordinate (after the deletion of $m-1$ coordinates) where $x,y,z$ triffer. One of $x_i,y_i,z_i$ has to be zero. Note that $y_i=0$ implies $z_i=x_i=0$, while $x_i=0$ if and only if $z_i=0$. Therefore, the triple $\{x,y,z\}$ does not triffer at position $i$, a contradiction.

    For the backward implication, suppose $C$ is $m$-minimal but not $m$-trifferent. Let $x,y,z$ be three distinct codewords and let $S$ denote the coordinates where $x,y,z$ triffer. Suppose $|S|\le m-1$. We extend $S$ to $\widetilde{S}$, a subset of $[n]$ with exactly $m-1$ coordinates. Without loss of generality, assume $z=0$ by replacing $x,y,z$ with $x-z,y-z,0$. Consider $w=x+y$. For any coordinate $i$ outside $S$, $\{x_i,y_i\}\neq\{1,2\}$. Therefore, $\sigma(w)\setminus \widetilde{S} = (\sigma(x)\setminus \widetilde{S}) \cup (\sigma(y)\setminus \widetilde{S})$. Since $x$ and $y$ remain distinct and nonzero after deleting the coordinates in $\widetilde{S}$ ($C$ is $m$-minimal), we have $(\sigma(x)\setminus \widetilde{S})\neq (\sigma(y)\setminus \widetilde{S})$, which implies that either $(\sigma(x)\setminus \widetilde{S})\subsetneq (\sigma(w)\setminus \widetilde{S})$ or $(\sigma(y)\setminus \widetilde{S})\subsetneq (\sigma(w)\setminus \widetilde{S})$, contradicting $m$-minimality.
\end{proof}

\subsection{The upper bound}

Let $\PG(k-1,q)$ denote the projective space $(\F_q^{k}\setminus\{0\})/\sim$, where $u\sim v$ if $u=\lambda v$.

\begin{mydef}
Let $m\ge 1$ be an integer. A set of points $S$ on a hyperplane $H$ is a spanning set \emph{with strength $m$}, if $S$ remains a spanning set after removing any $m-1$ of its points.
\end{mydef}

As introduced in \cite{davydov2011linear}, a \textit{strong blocking set} $B\subseteq \PG(k-1,q)$ is a set of points that intersects every hyperplane $H$ in a spanning set. See also \cite{heger2021short,alon2024strong,Bishnoi_trifference2024}.

\begin{mydef}
Let $m\ge 1$ be an integer. A set of points $B$ in $\PG(k-1,q)$ is a strong blocking set \emph{with strength $m$} if $B$ remains a strong blocking set after removing any $m-1$ of its points.
\end{mydef}

It has been shown in \cite{tang2021full} that $C$ is a minimal code of dimension $k$ in $\mathbb{F}_q^n$ with a generator matrix having no two linearly dependent columns if and only if the columns of its generator matrix form a strong blocking set in $\mathrm{PG}(k-1,q)$. We now provide a generalization of this statement.

\begin{lem}
    $C$ is an $m$-minimal code of dimension $k$ in $\mathbb{F}_q^n$ with a generator matrix having no two linearly dependent columns if and only if the columns of its generator matrix form a strong blocking set $S$ with strength $m$ in $\mathrm{PG}(k - 1, q)$.
\end{lem}
\begin{proof}
    A code has the $m$-minimal property if and only if after every deletion of $m - 1$ columns it is a minimal code. In other words, the strong blocking set formed by the columns of the generator matrix is a strong blocking set even after removing any set of $m-1$ elements. 
\end{proof}

In \cite{Bishnoi_trifference2024}, it was also shown that a set of points $\mathcal{L}\subseteq \PG(k-1,q)$ is a strong blocking set if and only if $B=\cup_{\ell\in \mathcal{L}} \ell\subseteq \mathbb{F}_q^k$ is an affine $2$-blocking set, i.e., if $B$ meets every $(k-2)$-dimensional affine subspace of $\mathbb{F}_q^k$ in at least one point. Hence we have the following corollary.

\begin{cor}
\label{cor:affine_blocking}
    Let $C$ be a linear code of dimension $k$ in $\mathbb{F}_3^n$ with generator matrix $G$.
    Suppose $G$ has no two linearly dependent columns. 
    Then $C$ is an $m$-trifferent code if and only if the union of $1$-dimensional subspaces of  $\mathbb{F}_3^k$ spanned by the columns of $G$ meet every affine subspace of dimension $(k - 2)$, not passing through the origin, in at least $m$ points. \qed
\end{cor}

Let $T_L(n, m)$ be the largest size of a linear $m$-trifferent code of length $n$.
\begin{thm}\label{thm:linearUpperBound}
For some absolute constant $c>0$, $T_L(n, m) \leq c\cdot 3^{\frac{n - m }{4.55}}$.
\end{thm}
\begin{proof}
Let $k = \log_3 T_L(n, m)$ be the dimension of the code. 
Then we have a set of $n$ points in $\mathrm{PG}(k - 1, 3)$ such that every $n - m + 1$ element subset of it is a strong blocking set. 
Since a strong blocking set in $\mathrm{PG}(k - 1, 3)$ has size at least $4.55 (k - 1)$ (see \cite{Bishnoi_trifference2024}), we get that
$n \geq 4.55(k - 1) + m - 1$, giving the statement. 
\end{proof}

\subsection{The lower bound}

From the analysis above, to lower bound the rate of $T_L(n,m)$, we only need to upper bound the size of $S$. We use a probabilistic approach to construct a strong blocking set $S$ with strength $m$ in the spirit of the bound of Héger and Nagy on the length of minimal codes \cite{heger2021short}.

\begin{thm}\label{strength}
    In $\PG(k-1,3)$, there exists a strong blocking set with strength $m$ of size at most $43k+18m-93$. In particular, when $k$ is sufficiently large and $m=o(k)$, the bound can be improved to $10k+5m-21$.
\end{thm}

\begin{proof}
We uniformly choose $t$ lines $\ell_1, \ell_2, \ldots \ell_{t}$ from $\PG(k-1,3)$. Consider the probability that $\B= \bigcup_{i=1}^{t} \ell_i$ fails to be a strong blocking set with strength $m$ in $\PG(k-1, 3)$. This happens if and only if there exists a hyperplane $H$ such that, after the removal of some $m-1$ points from $\B$ to obtain a modified set $\B^*$, the span of $H \cap \B^*$ is strictly contained in $H$, i.e., $\dim\langle H\cap\B^*\rangle<k-2$. In particular, there exists a $(k-3)$ dimensional projective subspace $H'$ such that apart from at most $m-1$ lines, $H'\cap \ell_i$ is non-empty. This happens with probability at most \[
    {k\brack k-2}_3\cdot \PP(\ell_1\cap H'\neq \emptyset)^{t-m+1}\cdot \binom{t}{m-1},
    \] which gives an upper bound on the probability that $\B$ fails to be a strong blocking set.

    Note that for $H'$ with $\dim H'=k-3$, \begin{align*}
    \PP(\ell_1\cap H'\neq \emptyset)&=\frac{{k-2 \brack 2}_3+\frac{1}{3}\cdot {k-2 \brack 1}_3\left({k\brack 1}_3-{k-2\brack 1}_3\right)}{{k\brack 2}_3}\\
    &=\frac{\frac{(3^{k-2}-1)(3^{k-3}-1)}{16}+\frac{1}{3}\cdot \frac{3^{k-2}-1}{2}\left(\frac{3^k-1}{2}-\frac{3^{k-2}-1}{2}\right)}{\frac{(3^{k}-1)(3^{k-1}-1)}{16}}\\
    &=\frac{(3^{k-2}-1)(3^{k-3}-1)+32\cdot 3^{2k-5}-32\cdot 3^{k-3}}{(3^k-1)(3^{k-1}-1)}\\
    &=\frac{11\cdot 3^{2k-4}-4\cdot 3^{k-1}+1}{27\cdot 3^{2k-4}-4\cdot 3^{k-1}+1}\\
    &<\frac{11}{27}.
    \end{align*} Also we can upper bound ${k\brack k-2}_3$ by $3^{2k-1}/16$. Hence, $\B$ is a blocking set with strength $m$ with positive probability if \[
    \frac{3^{2k-1}}{16}\cdot \left(\frac{11}{27}\right)^{t-m+1}\cdot \binom{t}{m-1}<1.
    \] Taking logarithms and relaxing, above inequality holds if \begin{equation}\label{eq:probInequality}
    (2k-1)\log 3-4+(t-m+1)\log \left(\frac{11}{27} \right) +t\cdot H\left(\frac{m-1}{t}\right)<0.
    \end{equation}

   A further relaxation of the entropy term yields the simpler sufficient condition 
\begin{equation}\label{eq:probInequality}
    (2k-1)\log 3-4 + (t-m+1)\log\left(\frac{11}{27}\right)+t<0.
    \end{equation} Rearranging this inequality gives \begin{equation}\label{eq:linearGrowth}
        t>-\frac{2\log 3}{\log\left(\frac{22}{27}\right)}k+\frac{\log\left(\frac{11}{27}\right)}{\log\left(\frac{22}{27}\right)}m+\frac{\log 3+4-\log\left(\frac{11}{27}\right)}{\log\left(\frac{22}{27}\right)}.
    \end{equation} Hence, one may choose \[
    t=\lceil 10.729k+4.385m-23.287\rceil.
    \] Since each line in $\PG(k-1,3)$ contains $4$ points, \[
    |\B|\le 4t \le 43k+18m-93.
    \]

    We now assume $k$ is sufficiently large and $m=o(k)$. From equation~\eqref{eq:linearGrowth} we see that $t$ grows at least linearly with $k$, hence $m=o(t)$. Consequently, the entropy term satisfies $H\left(\frac{m-1}{t}\right) \to 0$. A sufficient condition for inequality~\eqref{eq:probInequality} becomes for some $\epsilon>0$,
\[
\left(\log\left(\frac{11}{27}\right)+\epsilon\right) t<-2\log3 k+\log\left(\frac{11}{27}\right)m+\log 3+4-\log\left(\frac{11}{27}\right).
\] We may choose \[
    t=\lceil 2.447k+1.001m-5.310\rceil,
    \] which yields \[
    |\B|\le 4t \le 10k+5m-21.
    \]
\end{proof}

\begin{cor} For any integer $m<2n/9$, we have
 \[
    T_L(n,m)\ge 3^{\left\lfloor\frac{n-18m+93}{43}\right\rfloor}.
    \] Moreover, when $m=o(n)$, for sufficiently large $n$, this bound improves to \[
    T_L(n,m)\ge 3^{\left\lfloor\frac{n-5m+21}{10}\right\rfloor}.
    \]
\end{cor}
\begin{proof}

Let $C\subseteq \F_3^n$ be a linear code of dimension $k$. Its generator matrix defines a set of $n$ points in $\PG(k-1,3)$. As established earlier, $C$ is $m$-trifferent if and only if these points form a strong blocking set $\B$ with strength $m$.

By Theorem~\ref{strength}, there exists a strong blocking set $\B\subseteq \PG(k-1,3)$ with strength $m$ satisfying $|\B|\le 43k+18m-93$.
Thus, given $n\ge 43k+18m-93$, one can  select a collection of exactly a set of $n$ columns for a generator matrix that yields an $m$-trifferent code of dimension $k$. Since a linear code of dimension $k$ has $3^k$ codewords, it follows that $T_L(n,m)\ge 3^k$. Rewriting $n\ge 43k+18m-93$ as
\[
k\le \frac{n-18m+93}{43}
\]
shows that the choice
\[
k=\left\lfloor\frac{n-18m+93}{43}\right\rfloor
\]
is valid, and we deduce that
\[
T_L(n,m)\ge 3^{\lfloor \frac{n-18m+93}{43}\rfloor}.
\]

For the regime $m=o(n)$, note that the above argument shows $k$ grows linearly in $n$. Hence, we may apply the refined bound in Theorem~\ref{strength} and the same argument gives
\[
T_L(n,m)\ge 3^{\lfloor \frac{n-5m+23}{10}\rfloor}
\] for sufficiently large $n$.
\end{proof}

\section{Concluding remarks}\label{sec:concluding}

In this section, we discuss some related results to the $m$-trifference problem  and present computational results as well.

\subsection{Computational results and explicit constructions} 

Corollary~\ref{cor:affine_blocking} leads to an easy ILP formulation of the problem. Namely, 
we take a binary variable $x_i$ for the $i$-th $1$-dimensional subspace $V_i$ of $\mathbb{F}_3^k$, and for every $k - 2$ dimensional affine subspace $S$, we add the constraint $\sum_{i : V_i \leq S} x_i \geq m$.
Using Gurobi \cite{gurobi}, we have computed the following optimal values of these $m$-fold affine blocking sets.
\begin{center}
\begin{tabular}{c|ccccccc}
$(k, m)$ & $(3, 2)$ & $(4, 2)$ & $(4, 3)$ & $(4, 4)$ & $(4, 5)$ & $(4, 6)$ & $(4, 7)$\\\hline
$n$ & $12$ & $16$ & $24$ & $28$ & $30$ & $35$ & $38$
\end{tabular}
\end{center}

Via Corollary~\ref{cor:affine_blocking}, this table implies that $T_L(n, 2) = 3^3$ for all $12 \leq n \leq 15$, and $T_L(16, 2) = 3^4$.

For the non-linear case, direct computer computations give the following values: $T(4,2) = 4$, $T(5,2) = T(6,2) = 6$. Also, $T(4,3) = T(5,3) = 3$, $T(6,3) = T(7,3) = 4$.

\subsection{Explicit constructions}

Large linear trifferent codes were constructed via random methods \cite{Bishnoi_trifference2024}, so it was a relevant goal to obtain  good explicit constructions as well. Such codes were presented by Alon, Bishnoi, Das and Neri \cite{alon2024strong} using expander graphs. A similar approach can be applied in the general trifference problem as well.

For a graph $G$, let $\iota(G) = \min_{S \subseteq V(G)} \{|S| + \kappa(G - S)\}$, where
$\kappa(H)$ is the largest size of a connected component in $H$. 

Let $C$ be an infinite family of $[n, Rn, \delta n]_3$ codes, and let $G$ be an infinite family of $d$-regular graphs on $n$ vertices such that 
$\iota(G) \geq (1 - \delta)n + m - 1$. 
Then by the construction in \cite{alon2024strong} and the Lemma above, we get a linear $m$-trifferent code of length $(1 + d) n$ and dimension $k = Rn$. 

Note that constant degree expanders can be used to explicitly construct $d$-regular graphs $G$ on $n$ vertices with $\iota(G) \geq \left(1 - \frac{2 \sqrt{d}}{d + \sqrt{d} - o(1)} \right) n$, c.f. \cite{alon2024strong}.

\subsection{Relation to the sunflower conjecture}

Consider a code $C\subseteq \{0,1,2\}^n$ such that any three distinct codewords $x,y,z$ triffer in at least $1$ coordinate. Now consider the multiset
$$F(C)=\{\sigma(x): x\in C\}$$
where each $x$ is considered once. The elements of the multiset are in $\mathcal{P}([n])$.
If three distinct codewords $x,y,z$ triffer at position $i$ then $i$ appears in precisely two of the sets $\sigma(x), \sigma(y), \sigma(z)$. Observe that in $F(C)$, each word must have multiplicity at most $2$. This is because if we had $\sigma(x)=\sigma(y)=\sigma(z)$ for distinct $x,y,z$ then for every $i$, $i$ appears in 0 or 3 of the three sets, and so the three words cannot triffer at any position.

A set system $\mathcal{F}$ consisting of $k\ge 3$ sets is called a \textit{$k$-sunflower} if any two sets in the system have the same intersection, and a set system is called $k$-sunflower-free if no $k$ distinct elements of it form a $k$-sunflower.  A $3$-sunflower-free set-system is just called \textit{sunflower-free}. Naslund and Sawin \cite{naslundsawin_sunflower}  showed that a sunflower-free family $\mathcal{F}$ of subsets of $[n]$ must satisfy $|\mathcal{F}|\le \left(\frac{3}{2^{2/3}}\right)^{n(1+o(1))}$.

Consider the ground set of $F(C)$, denoting it by $\mathcal{F}$. Still in $\mathcal{F}$, it is true that there cannot be three sets $X,Y,Z$ such that for every $i\in [n]$, $i$ appears in either 0, 1 or 3 of the sets $X,Y,Z$. We can see that three sets satisfy this property precisely when they form a $3$-sunflower, therefore $\mathcal{F}$ must be a sunflower-free family, yielding $|C|\le 2|\mathcal{F}|\le \left(\frac{3}{2^{2/3}}\right)^{n(1+o(1))}$.

Here $\frac{3}{2^{2/3}}\approx 1.8899$. However, we already know that $T(n,1)\le Cn^{-2/5}\left(\frac32\right)^n$ from \cite{bhandari2025improved}, so this does not give an improvement.\\
Finally, let us mention that while this work focused on the generalization of the trifference problem, which was proven to be the most difficult one among $q$-hash codes, our methods are applicable for the $q$-hashing case as well.\\

\noindent\textbf{Acknowledgement.}
Part of the research reported in this paper was done at the 16th Emléktábla Workshop
(2024) in Vác, Hungary.

{\footnotesize
    \bibliographystyle{abbrv}
\bibliography{trifference.bib}}

\end{document}